\documentclass[ba,preprint]{imsart}

\usepackage{amsmath}
\usepackage{amssymb}
\usepackage{amsfonts}
\usepackage{amsthm}
\usepackage{array}

\usepackage{tikz,pgfplotstable}
\usetikzlibrary{patterns}
\pgfplotsset{compat=1.9} 
\usepackage{xcolor}

\DeclareMathOperator*{\argmax}{arg\,max}
\DeclareMathOperator*{\conv}{conv}

\newcommand{\R}{\ensuremath{\mathbb{R}}}
\newcommand{\eps}{\varepsilon}

\newcommand*{\horzbar}{\rule[.5ex]{2.5ex}{0.5pt}}

\newcommand{\der}{\text{\textup{d}}}
\newcommand{\diag}{\textup{diag}}

\providecommand{\x}{}
\renewcommand{\x}{\mathbf{x}}
\newcommand{\y}{\mathbf{y}}

\newcommand{\hil}{\mathcal{H}}
\newcommand{\hilp}{\mathcal{H}_p}
\newcommand{\hilo}{\mathcal{H}_o}
\newcommand{\obs}{\mathcal{O}}

\newcommand{\fwd}{\mathcal{F}}

\newcommand{\obsm}{\widehat{\obs}}
\newcommand{\Sigmam}{\widehat{\Sigma}}
\newcommand{\postcovm}{\widehat{\Gamma_{\textup{post}}}}

\newcommand{\tar}{\Psi}

\newcommand{\data}{\mathbf{d}}
\newcommand{\param}{\mathbf{m}}
\newcommand{\normal}{\mathcal{N}}
\newcommand{\pr}{\mu_{\textup{pr}}} 
\newcommand{\post}{\mu_{\textup{post}}} 
\newcommand{\prmean}{\param_{\textup{pr}}} 
\newcommand{\postmean}{\param_{\textup{post}}} 
\newcommand{\postcov}{\Gamma_{\textup{post}}} 
\newcommand{\prcov}{\Gamma_{\textup{pr}}} 
\newcommand{\modcov}{\Gamma_{\textup{model}}} 
\newcommand{\tmp}{\mathcal{G}}
\newcommand{\meas}{\mathbf{o}}
\newcommand{\ev}{\mathbf{e}} 
\newcommand{\func}{\mathbf{a}}
\newcommand{\tr}[1]{\textup{tr}\left \{#1 \right \} }
\newcommand{\ttr}[1]{\textup{tr}\ #1}
\newcommand{\rank}{\textup{rank}\ }

\newcommand{\opt}{\mathcal{D}}
\newcommand{\postopt}{\mu_{\textup{post}}^{\data, \opt}} 

\RequirePackage{amsthm,amsmath,amsfonts,amssymb}
\RequirePackage[numbers]{natbib}
\RequirePackage[colorlinks,citecolor=blue,urlcolor=blue,backref=page,backref=page]{hyperref}
\RequirePackage{graphicx}

\startlocaldefs
\theoremstyle{plain}

\newtheorem{theorem}{Theorem}[section]
\newtheorem{lemma}[theorem]{Lemma}
\newtheorem{proposition}[theorem]{Proposition}
\newtheorem{corollary}[theorem]{Corollary}
\theoremstyle{definition}
\newtheorem{definition}[theorem]{Definition}

\newtheorem*{ctheorem}{Theorem}
\theoremstyle{remark}

\endlocaldefs

\begin{document}

\begin{frontmatter}
  \title{Clusterization in D-optimal designs: the case against linearization}
  \runtitle{Clusterization in D-optimal designs}

\begin{aug}
\author[A]{\fnms{Yair}~\snm{Daon}\ead[label=e1]{firstname.lastname@gmail.com}\orcid{0000-0002-7491-2008}}
\address[A]{Azrieli Faculty of Medicine, Bar-Ilan University, Safed, Israel 
\printead[presep={,\ }]{e1}}
\end{aug}

\begin{abstract}
  Estimation of parameters in physical processes often demands costly
  measurements, prompting the pursuit of an optimal measurement
  strategy. Finding such strategy is termed the problem of optimal
  experimental design, abbreviated as optimal design. Remarkably,
  optimal designs can yield tightly clustered measurement locations,
  leading researchers to fundamentally revise the design problem just
  to circumvent this issue. Some authors introduce error correlation
  among error terms that are initially independent, while others
  restrict measurement locations to a finite set of locations. While
  both approaches may prevent clusterization, they also fundamentally
  alter the optimal design problem.

  In this study, we consider Bayesian D-optimal designs, i.e.~designs
  that maximize the expected Kullback-Leibler divergence between
  posterior and prior. We propose an analytically tractable model for
  D-optimal designs over Hilbert spaces. In this framework, we make
  several key contributions: (a) We establish that
  measurement clusterization is a generic trait of D-optimal designs
  for linear inverse problems with independent Gaussian measurement
  errors and a Gaussian prior. (b) We prove that introducing
  correlations among measurement error terms mitigates
  clusterization. (c) We characterize D-optimal designs as
  reducing uncertainty across a subset of prior covariance
  eigenvectors. (d) We leverage this characterization to
  argue that measurement clusterization arises as a consequence of the
  pigeonhole principle: when more measurements are taken than there
  are locations where the select eigenvectors are large and others are
  small --- clusterization occurs. Finally, we use our analysis to
  argue against the use of Gaussian priors with linearized physical
  models when seeking a D-optimal design.
\end{abstract}



\end{frontmatter}

\section{Introduction}\label{section:intro}
Experimental design is a systematic approach to specifying every
possible aspect of an experiment \cite{chaloner1995}. In this study,
we focus on experiments that involve measuring quantities arising from
naturally occurring processes. For us, therefore, experimental design
entails determining which specific measurements to record from a given
process. In a typical experiment the number of measurements that can
be recorded is limited by time, cost and other constraints. Therefore
selecting the right set of measurements to record is crucial as it
directly impacts the utility of the data collected, and consequently,
the validity of the conclusions drawn.

For example, consider the process of searching for oil: measurements
require digging deep holes in the ground, known as "boreholes". A
\emph{design} simply indicates which boreholes should be dug
\cite{horesh2008borehole}. Since digging boreholes is expensive, only
a limited number of them can be dug, and thus carefully choosing their
locations is crucial.

Specifying the right set of measurements holds particular significance
when solving an \emph{inverse problem} --- i.e.~making inference of
the physical world utilizing a physical model of a phenomenon of
interest \cite{tarantola2005,kaipio2005}. In an inverse problem, we
seek to infer physical quantities, based on observations and a model
of the physical world. Models are typically phrased in the language of
ordinary or partial differential equations.

Inverse problems are ubiquitous in science and engineering. For
example, in electrical impedance tomography (EIT) we seek to infer the
structure of the inside of a human body. Inference is conducted based
on observations of electrical current impedance measured at specific
electrode locations on the skin, leveraging known equations of
electric current flow \cite{horesh2010impedance}. Magnetic Resonance
Imaging (MRI) also entails solving an inverse problem. There,
radio-frequency pulses are sent through the human body in the presence
of a strong magnetic field. The response of one's body contents to
these radio-frequency pulses is measured, allowing a radiologist to
view the internals of a body in a noninvasive manner
\cite{horesh2008mri}. An inverse problem also arises in oil and gas
exploration, where acoustic waves are sent through "boreholes" ---
deep cylindrical wells drilled into the ground. Data of travel and
return times of the acoustic waves are recorded. Wave travel and
return time is influenced by the properties of the subsurface
materials, such as density, elasticity, and the presence of fluids or
voids. Combining travel time data with a geophysical model of the
contents of the earth's crust facilitates reconstructing the structure
of the subsurface in a process called \emph{borehole tomography}
\cite{horesh2008borehole}. Inverse problems also arise in many other
areas of seismology and geology \cite{rabinowitz1990, steinberg1995}
and medical imaging \cite{tarantola2005}. In many inverse problems the
goal is to infer some numerical parameter. However, in the formulation
we consider in this study, as well as in the examples above, the goal
is to conduct inference over some \emph{function} over $\Omega$, where
\(\Omega \subseteq \mathbb{R}^d, d=1,2,3\) is a spatial domain of
interest.

\subsection{D-optimal Designs}\label{subsec:D}
In all of the examples of inverse problems we have seen, as well as in
many other applications, only a limited number of measurements are
allowed. For example, in borehole tomography, a measurement involves
drilling a deep well in the ground --- a very costly endeavor. In
medical imaging applications like MRI, the time allotted for each
patient in the MRI machine limits the number of measurements that can
be taken during a single session. Therefore, measurements should be
chosen to extract as much information from the experiment. Typically,
a user will consider some utility, called a \emph{design criterion}
and take measurements that maximize this utility. Two of the most
widely recognized and extensively studied design criteria are the A-
and \emph{D-optimality} criteria.

In this study, we focus on the Bayesian D-optimality
criterion. Bayesian D-optimality carries a simple and intuitive
meaning: measurements chosen according to the D-optimality criterion
maximize the expected Kullback-Leibler (KL) divergence from posterior
to prior \cite{chaloner1995, AlexanderianGloorGhattas14,
  CoverThomas91}.
Recall that for a discrete parameter $\param$ and data $\data$, the KL
divergence is defined as
\begin{equation}\label{eq:basic_KL}
  D_{KL}\left (\Pr(\param|\data)||\Pr(\param)\right ) =  \sum_{\param} \log
  \frac{\Pr(\param|\data)}{\Pr(\param)} \Pr(\param|\data). 
\end{equation}
Of course, data $\data$ is not known before the experiment is
conducted. Hence, we average over $\data$ to define the D-optimality
criterion as:
\begin{equation*}
  \mathbb{E}_{\data}\left [ D_{KL}\left (\Pr(\param|\data)||\Pr(\param)\right ) \right ].
\end{equation*}
We refer to a set of measurements that maximizes the D-optimality
criterion as a \emph{D-optimal design}.

For a linear model with Gaussian prior in finite dimensions, a
D-optimal design minimizes the determinant of the posterior covariance
\cite{chaloner1995}. In Section \ref{subsec:D_optimal_design} we give
a more general definition of the D-optimality criterion that applies
to arbitrary measures. We also show how the D-optimality criterion
generalizes to linear models over infinite-dimensional Hilbert spaces
(e.g.~function spaces).

\subsection{A toy model: the 1D heat equation}\label{subsec:toy}
For concreteness, let us now consider a toy inverse problem: inferring
the initial condition for a partial differential equation known as the
\emph{heat equation} in one dimension (1D heat equation
henceforth). Readers less familiar with partial differential equations
can think of the following setup: before us there is a metal rod,
perfectly insulated except for its tips. At each point on the rod the
temperature is different and unknown, and we call this temperature
distribution the \emph{initial condition}. We wait for a short time
$T$ and let the heat dissipate a bit inside the rod. The heat equation
determines the heat distribution inside the rod at time $t=T$.

The full time evolution of heat in the rod $\Omega=[0,1]$ is formally
described by the following three equations:
\begin{subequations}
  \begin{alignat}{2}
    u_t &= \Delta u &&\qquad \text{in } [0,1] \times [0,\infty), \label{eq:heat1}\\
    u &= 0 &&\qquad \text{on } \{0, 1\} \times [0,\infty), \label{eq:heat2}\\
    u &= u_0 &&\qquad \text{on }[0,1] \times \{0\}. \label{eq:heat3}
  \end{alignat}
\end{subequations}
As time passes, heat dissipates across the rod as hotter regions,
e.g.~local maxima of the heat distribution, become cooler. This
behavior is captured by eq.~\eqref{eq:heat1}: the Laplacian at local
maxima is negative, so $u_t = \Delta u$ implies $u$ decreases at its
maxima, and the reverse happens at local minima. Eq.~\eqref{eq:heat2}
describes the \emph{boundary condition}, which dictates how heat
interacts with the outside world at the two edges of the rod
$\{0,1\}$. Specifically, eq.~\eqref{eq:heat2} implements an absorbing
boundary, i.e.~heat that interacts with the boundary immediately
disappears. This type of boundary condition is known as a
\emph{homogeneous Dirichlet boundary condition}. Eq.~\eqref{eq:heat3}
describes the \emph{initial condition}, i.e.~the heat distribution
inside the rod at time $t=0$.

Solving the 1D heat equation is straightforward. Consider an initial
heat distribution $u_0\in L^2([0,1])$ that satisfies the homogeneous
Dirichlet boundary condition eq.~\eqref{eq:heat2}. This initial
condition $u_0$ is a linear combination of sines $\ev_n(x) = \sin(\pi
n x)$, so $u_0 = \sum_{n\geq 1} a_n \ev_n$ for some $\{a_n\}_{n\geq
  1}$. Now, note that $\ev_n$ are in fact eigenvectors of the
Laplacian: $\Delta \ev_n = -\pi^2 n^2\ev_n$. It should not be hard to
believe that $u(\cdot, T) = \sum_{n\geq 1} a_n \exp(-\pi^2 n^2T )
\ev_n$ \cite{renardy2006PDE}. Thus, eigenvectors of the linear
operator that describes the time evolution of heat are $\ev_n(x) =
\sin(\pi n x)$ with corresponding eigenvalues
\begin{equation}\label{eq:decay}
  \exp(-\pi^2 n^2T ).
\end{equation}

In the \emph{inverse problem of the 1D heat equation}, our goal is to
infer the initial condition $u_0 = u(\cdot, 0)$ from noisy
observations of the final state $u(\cdot, T)$. However, we are not
able to measure $u(\cdot, T)$ at every point in the domain
$\Omega$. Rather, we take some number of temperature measurements on
the rod once our waiting time $T$ has passed and try to infer
$u(\cdot, 0)$ from these measurements.

The inverse problem of the 1D heat equation is difficult ("ill-posed")
since heat spreads in a diffusive manner and any roughness in the
initial condition is quickly smoothed, as implied by the squared
exponential decay of eigenvalues in eq.~\eqref{eq:decay} with
$n\to\infty$. See Supplementary movies S1 and S2 for an illustration
of this phenomenon\footnote{Code generating these movies is located in
module \texttt{movies.py} in the accompanying
\href{https://github.com/yairdaon/OED}{repository}.}. For ill-posed
problems regularization is required, and we implement such
regularization via a fully Bayesian formulation of the inverse
problem.

In order to find a fully Bayesian formulation for our inverse problem
we first need to specify a prior on \emph{functions} defined over
$\Omega = [0,1]$. For the sake of simplicity we would like to utilize
a Gaussian prior. While specifying the prior mean is easy --- we take
the zero function on $\Omega$ --- specifying a prior covariance is
more involved. In analogy with the finite dimensional case, we seek a
covariance \emph{operator} that is positive definite and imposes
sufficient regularity on functions over $\Omega$. Thus, we set our
prior for the initial condition $u_0 \sim \mathcal{N}(0,
(-\Delta)^{-1})$, where $\Delta$ is again defined with a homogeneous
Dirichlet boundary condition. The intuition here is that since
$\Delta$ is a differential operator, it is "roughing" (the opposite of
smoothing) and so $(-\Delta)^{-1}$ is smoothing. Prior realizations
are generated, in analogy to the finite dimensional case, by smoothing
white noise: $(-\Delta)^{-1/2}\mathcal{W}$, where $\mathcal{W}$ is
white noise \cite{rue2011}. Our choice of prior covariance operator
ensures the posterior is well-defined and prior realizations are
well-behaved, see Theorem 3.1 and Lemma 6.25 in \cite{Stuart10} for
details. The second ingredient of a fully Bayesian formulation of the
inverse problem is likelihood. We simply assume centered iid Gaussian
measurement error model, giving rise to a the widely applicable
Gaussian likelihood function.

Now that we have a fully Bayesian formulation of our toy inverse
problem, we would like to find a D-optimal design. Defining D-optimal
designs over function spaces is somewhat mathematically
involved. Thus, to keep our discussion here focused, we will explore
these technical details later in Section
\ref{subsec:D_optimal_design}. For now, it suffices to know that
D-optimal designs can be defined for our problem.

\subsection{Measurement Clusterization}
Surprisingly, A- and D-optimal designs for inverse problems have been
observed to yield remarkably similar measurements in certain cases
\cite{fedorov1996, nyberg2012, fedorov1997, Ucinski05,
  neitzel2019sparse}. This phenomenon is illustrated for our toy
inverse problem in Fig.~\ref{fig:clusterization_illustration}, where
D-optimal measurement locations are shown for different numbers of
measurements. Notably, for six measurements, a D-optimal design yields
two sets of measurements that are identical. Following
\cite{Ucinski05}, we refer to this intriguing phenomenon as
\emph{measurement clusterization}. We consider a design to be
\emph{clustered} when two or more of its constituent measurements are
identical.

\begin{figure}
    \centering
    \includegraphics[height=0.5\textwidth]{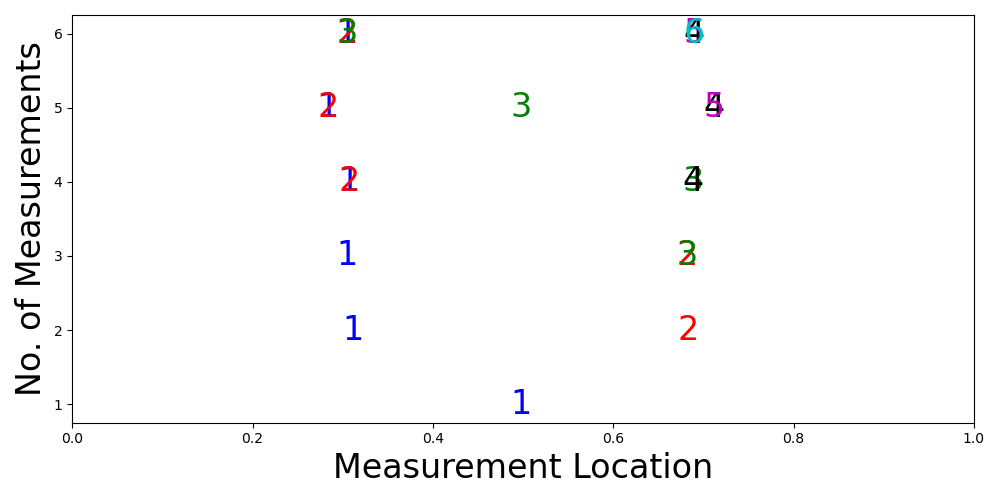}
    \caption{Measurement clusterization in D-optimal designs for the
      inverse problem of the 1D heat equation. Measurement locations
      were chosen according to the Bayesian D-optimality criterion of
      Theorem \ref{thm:d_optimality}. Measurement locations are
      plotted over the computational domain \(\Omega = [0, 1]\)
      (x-axis), for varying numbers of measurements (y-axis). The
      colored numbers are measurement indices, plotted for visual
      clarity. Measurement clusterization already occurs for three
      measurements: the second measurement (red) is overlaid on the
      third (green). For five measurements, first (blue) and second
      (red) measurements are clustered, as well as the fourth (black)
      and the fifth (magenta).}
  \label{fig:clusterization_illustration}
\end{figure}

Clusterization should not be confused with replication. Replication
requires that the experimentalist executes multiple trials under
circumstances that are \emph{nominally identical} \cite[Section
  1.2.4]{morris2011}. Replication is commonly viewed as a beneficial
and even necessary aspect of optimal experimental design
\cite{fisher1949design, morris2011, schafer2001replication}.
For example, \cite{fisher1949design}, suggested repeating his famous
milk and tea experiment in order "to be able to demonstrate the
predominance of correct classifications in spite of occasional
errors". Unfortunately, in the experiments we consider, replication is
impossible. For example, in the MRI problem, we cannot generate an
individual nominally identical to the one we wish to scan.

Similarly to a design implementing replication, a clustered design
reduces the signal-to-noise ratio of the repeated measurements
\cite{telford2007brief}. The difference is that a clustered design
takes repeated measurements at the expense of other quantities not
measured at all. For example, consider an experiment measuring the
effect of rainfall on grass growth \cite{fay2000rainfall}. The
experiment involved four rainfall manipulation "treatments"
(i.e.~simulating different timing and quantity of rainfall), each
replicated three times over different plots of land. Indeed, it seems
reasonable for researchers to replicate the phenomenon they are trying
to study. A clustered design in such an experiment would imply the
researchers should take a repeated measurement \emph{on the same
plot}, at the expense of measuring grass growth in other plots!

As another example illustrating the differences between replication
and clusterization, consider an experiment involving a chemical
reaction conducted at different temperatures. An investigator may
choose to repeat the experiment at $10^\circ C$ and $30^\circ C$,
giving rise to \emph{replication}. Within a single experiment, the
investigator may also sample the concentrations of the reagents
through the course of the experiment. Simultaneous sampling (within a
single experiment) gives rise to \emph{clusterization}. While
replication is intuitively sensible for this chemical experiment,
clusterization is not.

To conclude our short discussion on replication vs.~clusterization:
these are fundamentally different concepts and even though replication
is quite intuitive, it is inapplicable to the inverse problems we
consider in this manuscript.

Measurement clusterization is generally considered undesirable
\cite{fedorov1996, nyberg2012, fedorov1997, Ucinski05,
  neitzel2019sparse}. It is counterintuitive for a so-called "optimal"
design to prioritize marginally improving accuracy in a small region
over taking measurements in unexplored regions of the computational
domain, where they could significantly reduce uncertainty.

The issue with clusterization is evident in
Fig.~\ref{fig:clusterization_illustration}, where most clustered
designs shown place little to no measurement weight in the center of
the domain, where prior uncertainty is largest. It seems as though a
clustered design assumes that focusing measurement efforts on a small
subset of the domain can somehow yield information about the entire
domain. Intuitively, this is an unreasonable assumption, albeit it
holds for \emph{linear} models.

The view that clusterization is undesirable prompted the exploration
of various remedies to address this issue. One approach involves
merging close measurements \cite{fedorov1997}; however, this strategy
merely overlooks the phenomenon of measurement clusterization. An
alternative solution lies in \emph{clusterization-free design}s, where
measurement locations are deliberately chosen to be distant from one
another. This can be achieved by imposing distance constraints between
measurements or by introducing correlated errors that account for both
observation error and model misspecification \cite{Ucinski05}. For
instance, in the context of time-series analysis for pharmacokinetic
experiments, measurement clusterization can be mitigated by
incorporating the modeling of auto-correlation time within the noise
terms \cite{nyberg2012}.

In spatial problems involving choice of measurements within a domain
\(\Omega \subseteq \mathbb{R}^d, d=1,2,3\), many researchers
circumvent the problem of measurement clusterization by restricting
measurements to a coarse grid in \(\Omega\) \cite{koval2020,
  alexanderian2021, attia2022, alexanderian2014, alexanderian2016,
  alexanderian2018efficient, brunton2016}. This approach incurs a
significant computational cost as it requires solving a difficult
combinatorial optimization problem for measurement locations over a
discrete set. The combinatorial optimization problem is usually
relaxed by first assigning optimal measurement weights in
\(\mathbb{R}_+\) to the potential measurement locations. Some
researchers incorporate a sparsifying \(\ell_1\) penalty term into the
design criterion, which is subsequently thresholded to achieve the
desired binary design over the coarse grid
\cite{horesh2008borehole}. Others progressively relax the \(\ell_1\)
penalty to an \(\ell_0\) penalty via a continuation method
\cite{alexanderian2016, alexanderian2014}. Others cast the problem of
finding optimal measurement weights as a stochastic optimization
problem \cite{attia2022stochastic}. All of the aforementioned methods
may indeed find a binary optimal design restricted to a given coarse
grid. However, none addresses one fundamental issue: the restriction
of measurement locations to a coarse grid in \(\Omega\) fundamentally
changes the optimal design problem and thus results in a sub-optimal
design.

Avoiding measurement clusterization is a pragmatic approach:
intuitively, researchers recognize that measurement clusterization is
undesirable, even though the underlying reasons may not be fully
clear. Consequently, they strive to prevent it and devise various
methodologies to avoid it. Yet each and every one of these
methodologies achieves its objective by imposing restrictions on
measurement locations, thereby fundamentally altering the optimal
design problem. To the best of my knowledge, no previous study has
tried to address some seemingly simple yet fundamental questions:
Why does imposing correlations between observations alleviate
measurement clusterization?
Is measurement clusterization a generic phenomenon? 
And, most importantly: Why does measurement clusterization occur?
%
%

\subsection{Contribution}
The primary objective of this study is to provide a deep understanding
of measurement clusterization by addressing the aforementioned
questions. Our focus centers around investigating the Bayesian
D-optimality criterion. We conduct an analysis of Bayesian D-optimal
designs within the context of linear inverse problems over Hilbert
spaces and study two inverse problems: (a) In Sections
\ref{section:prelim} and \ref{section:D_and_grad} we propose a novel
generic model for an inverse problem where D-optimality maintains
analytical tractability and D-optimal designs are identified via
Lagrange multipliers. This analytical framework facilitates the
exploration of the questions posed at the end of the previous
paragraph. We also study (b) the inverse problem of the 1D heat
equation from Section \ref{subsec:toy} above. Investigating both
inverse problems allows us to answer the questions posed in the
previous section:

\begin{enumerate}
\item \label{q:generic} \textbf{Is measurement clusterization a
  generic phenomenon?}
  We give two complementing answers to this question. First, from a
  theoretical perspective, we show that clusterization mainly depends
  on how quickly the eigenvalues of the prior covariance in
  observation space decay.
  See Section \ref{section:vanishing} --- particularly, Theorem
  \ref{thm:char} and the discussion following it. Furthermore, in
  Section \ref{subsec:lemma_sims} we show results of numerical
  experiments, where simulations of our model give rise to D-optimal
  designs that exhibit clusterization with high probability. Thus,
  given the genericity of our model, we expect measurement
  clusterization to be a generic and ubiquitous phenomenon.

\item \label{q:mitigate} \textbf{Why does imposing correlations
  between observations alleviate measurement clusterization?} In
  Section \ref{section:non_vanishing}, we rigorously demonstrate the
  role of model error in mitigating clusterization, thereby
  corroborating earlier observations made by various
  researchers. Specifically, our proof shows that identical
  measurements result in no gain in design criterion when observation
  error amplitude tends to zero. Moreover, in Section
  \ref{subsec:corr_errors_sims}, we show that an error term
  corresponding to correlations between measurements mitigates
  clusterization in the inverse problem of the 1D heat equation.

\item \label{q:why} \textbf{Why does measurement clusterization
  occur?} In Sections \ref{subsec:why} we give two compelling answers
  to this question by (i) transporting insights we gain from our
  generic model to the inverse problem of the 1D heat equation, and
  (ii) connecting measurement clusterization to Carath\'eodory's
  Theorem.

  Our analysis for the heat equation relies on conclusions from our
  generic model. In particular, Theorem \ref{thm:char} reveals that a
  D-optimal design focuses on a select set of prior eigenvectors,
  i.e.~those with the largest eigenvalues in the prior covariance
  spectrum. In practical scenarios, the number of locations where (a)
  all relevant prior eigenvectors are significantly large, and (b)
  other eigenvectors are close to zero, is limited. Consequently, the
  clusterization of measurements arises as a natural consequence of
  the pigeonhole principle, as there are more measurements available
  than there are locations satisfying conditions (a) and (b).

  The connection to Carath\'eodory's Theorem also builds on the focus
  of measurement effort to the abovementioned select set of prior
  eigenvectors. This allows us to move from an infinite-dimensional
  setting to finite dimensions, where the conditions of
  Carath\'eodory's Theorem hold. We conclude that a D-optimal design
  arises by weighting a small number of measurements. If the number of
  allowed measurements is larger than the small number dictated by
  Carath\'eodory's Theorem, we have excessive weight on some
  measurements, which can be interpreted as clusterization.
  
\end{enumerate}

\subsubsection{Implications}\label{subsub:implications}
Our answer to Question \ref{q:generic} implies that encountering
clusterization could be expected in many different problems across
many different scientific fields. Researchers that encounter
clusterization should not be surprised or wary. In Our answer to
Question \ref{q:why}, we explain what our view of the cause of
clusterization is. It appears, the cause is generic: a D-optimal
design reduces uncertainty for a select set of prior covariance
eigenvectors --- those with the most prior uncertainty, i.e.~those the
practitioner cares about the most!

A further implication of the analysis we present is that
clusterization can serve as an evidence to the number of relevant
eigenvectors in the problem. These eigenvectors correspond to the
relevant degrees of freedom in the problem. Thus, clusterization tells
us that we might be able to reduce the complexity of a computational
model, e.g.~by dropping discretization points.

One potential cause of clusterization is our choice of prior. Gaussian
priors, coupled with a Gaussian likelihood and a linear forward
problem give rise to a closed form solution for the posterior via
conjugacy. As we show in Section \ref{subsec:lemma_sims}, such
assumptions generically give rise to clusterization. While an
assumption of a Gaussian likelihood (i.e.~independent Gaussian noise)
is standard, and rooted in the central limit theorem, an assumption of
Gaussian prior is merely a matter of convenience. Therefore, we advise
any practitioner who encounters clusterization to replace their
Gaussian priors with non-Gaussian priors instead \cite{hosseini2017,
  hosseini2019}, as these are not only more realistic but also
expected to mitigate clusterization.

Another potential cause of clusterization is linearization. While in a
real applications the model is not necessarily linear, some authors
consider a linearized version of their model when seeking optimal
designs \cite{fedorov1996, neitzel2019sparse}. Our analysis then shows
that the linearity of the forward problem is also an important
ingredient in giving rise to clustered designs. Therefore, our advice
to practitioners is to avoid linearization, and find D-optimal designs
via other methods, e.g.~the sampling method of \cite{ryan2003}.

Given the above discussion, it is our belief that the methods
suggested by other authors to avoid clusterization are merely
overlooking the problem. We do not view clustered designs as
inherently bad on their own. Rather, we suggest that if a clustered
design arises, a practitioner should revisit their modeling choices
--- specifically, their choice of prior and model linearization. If
the practitioner is confident in their analysis, then clustered
designs should be avoided only to the extent necessitated by the
physical measuring apparatus.

Clusterization in D-optimal designs raises the question of whether
D-optimal designs should be pursued at all. While other authors have
established convergence results for decaying measurement error
\cite{knapik2011}, space-filling measurements \cite{teckentrup2020}
and randomly sampled designs \cite{nickl2023}, their results do not
hold for D-optimal designs. Luckily, we can expect convergence from
D-optimal designs (including clustered designs) in the following
sense: the posterior uncertainty ellipsoid will contract to zero along
each one of its eigenvectors. See Section \ref{subsub:convergence} for
a precise statement and a proof based on Theorem \ref{thm:char} ---
the main theorem of this manuscript.

\subsubsection{Other Contributions}
In Theorem \ref{thm:char}, we also show that D-optimal designs are
best understood in the space of \emph{observations}
This is in accordance with previous work by \cite{koval2020}, who
showed that A-optimal designs are best constructed in the space of
observations.

In the process of proving Theorem \ref{thm:char} we prove and
generalize several lemmas. Among those, is Lemma \ref{lemma:free},
which is (to the the best of my knowledge) novel: We decompose a
symmetric positive definite matrix \(M \in \mathbb{R}^{k \times k}\)
with \(\ttr M = m \in \mathbb{N}\) as \(M = AA^t\), where \(A\in
\mathbb{R}^{k \times m}\) has unit norm columns.

\subsection{Limitations}\label{subsec:limitations}
The main limitation of this study is that our generic model does not
correspond to any specific real-life problem. Specifically, in its
current form, our model does not allow point evaluations. Thus, while
our model is generic enough to be analytically tractable, one may
argue that our model is too far removed from any real application. To
these claims I would answer that scientists have a long history of
studying models that are bare-bones simplifications of real systems,
e.g.~the Ising model \cite{cipra1987}, the Lorenz system \cite{brin},
the Lotka-Volterra equations \cite{logan2006}, the Carnot engine
\cite{kardar2007} and many others.

\section{Preliminaries and Notation}\label{section:prelim}
Here we present the basics of Bayesian inverse problems over Hilbert
spaces. Since we are ultimately interested in inferring a function
over some domain, we will keep in mind that these Hilbert spaces
should really be thought of as function spaces, utilizing a setup
similar to \cite{knapik2011}. A deeper treatment of the foundations
for inverse problems over function spaces can be found in
\cite{Stuart10}.

\subsection{Forward Problems}\label{subsec:abstract_OED}
In this section we give definitions and notations for forward
problems, which are an essential part of the inverse problems we
discuss later. Consider a "parameter space" \(\hilp\) and an
"observation space" \(\hilo\) --- both separable Hilbert spaces (the
subscripts p and o are for "parameter" and "observation",
respectively). The parameter space includes the quantity we seek to
infer; in the inverse problem of the heat equation, the parameter
space $\hilp$ is where the initial condition lives. The observation
space $\hilo$, on the other hand, is the space from which we take
measurements; in the example of the heat equation, $u(\cdot, T) \in
\hilo$.

The connection between parameter and observation spaces is made by the
\emph{forward operator} \(\fwd: \hilp \to \hilo\). We assume the
forward operator \(\fwd\) is linear. In the inverse problem of the 1D
heat equation, the forward operator is determined by the time
evolution of the 1D heat equation \eqref{eq:heat1} and
\eqref{eq:heat2}, so $u(\cdot, T) = \fwd u_0$.

Measurements are taken via a linear \emph{measurement operator}
\(\obs\), which is the concatenation of linear functionals
$\meas_1,\dots,\meas_m$ we call \emph{measurements}:
\begin{equation*}
  \obs u = (\meas_1(u), \dots, \meas_m(u) )^t \in \R^m,\ u \in \hilo.
\end{equation*}
Thus, \(\obs \in ( \hilo^* )^m\), where \(m\) is the number of
measurements taken\footnote{The alert reader will likely ask how do we
reconcile point measurements $\delta_x$ as suggested by the
formulation of the 1D heat equation with working in Hilbert spaces. We
don't. We follow standard practice in the literature and restrict our
analysis to Hilbert spaces. We can satisfy ourselves with the fact
that point evaluations could be approximated in a standard Hilbert
space like $L^2(\Omega)$.}. 

Data is acquired via noisy observations, and we consider two types of
error terms: Spatially correlated model error \(\eps' \sim
\normal(0,\modcov)\) with \(\modcov\) a covariance operator; and
observation error denoted \(\eps \sim \normal(0, \sigma^2 I_m)\), with
\(I_m \in \mathbb{R}^{m \times m}\) the identity. Both error terms and
the prior (see Section \ref{subsec:bayesian_inverse_problems} below)
are assumed independent of each other. Thus, data is acquired via
\begin{align}\label{eq:inverse_problem}
  \data := \obs (\fwd \param + \eps') + \eps = \obs \fwd \param + \obs \eps' + \eps.
\end{align}

It is easy to verify that \(\obs \eps' + \eps \in \R^m\) is a centered
Gaussian random vector with covariance matrix

\begin{align}\label{eq:Sigma}
  \begin{split}
    \Sigma(\obs) :&= \mathbb{E}[ (\obs \eps' + \eps) (\obs \eps' +
      \eps)^t ]
    = \obs \modcov \obs^* + \sigma^2I_m , 
  \end{split}
\end{align}
where
\begin{align}\label{eq:modcov_explained}
  \begin{split}
    [\obs \modcov \obs^*]_{ij} & = e_i^t \obs \modcov \obs^* e_j 
    = \meas_i (\modcov \meas_j).
  \end{split}
\end{align}
Taking \(\modcov = 0\) is a common practice
\cite{tarantola2005,kaipio2005,Vogel02} and then \(\Sigma =
\sigma^2I_m\) is a scalar matrix which does not depend on \(\obs\).

\subsection{Bayesian Linear Inverse Problems}\label{subsec:bayesian_inverse_problems}
In the previous section, we saw how a parameter $u\in \hilp$ is
transported to the observation space via the forward operator $\fwd u
\in \hilo$, how observations are generated from a parameter via $\obs
\fwd u$ and how observations and noise give rise to data $\data$. It
is time to formulate the process of inferring the parameter as a
Bayesian inverse problem. We have already defined the Gaussian
likelihood in the previous section, and now we will define the prior.

We take a Gaussian prior \(\param \sim \pr = \normal(\prmean
,\prcov)\) with some appropriate covariance operator \(\prcov\) on
\(\hilp\) \cite{Stuart10}. For example, for the inverse problem of the
1D heat equation we chose $\prcov = (-\Delta)^{-1}$, as described in
Section \ref{subsec:toy}. Note that \(\fwd \prcov \fwd^*\) is the
prior covariance in \(\hilo\) \cite{Stuart10}, and as such is assumed
invertible --- an assumption which we will use later (if \(\fwd\) has
a nontrivial kernel we utilize Occam's Razor and ignore said kernel
altogether).

Since $\fwd$ is linear and $\pr$ is Gaussian --- the posterior
\(\post\) is Gaussian as well. We do not utilize the posterior mean in
this study, but the posterior covariance operator $\postcov$ is given
by the known formula \cite{Stuart10}:
\begin{align}\label{eq:postcov}
  \postcov = (\prcov^{-1} + \fwd^* \obs^* \Sigma^{-1} \obs \fwd
  )^{-1}.
\end{align}

\subsection{Bayesian D-Optimal Designs}\label{subsec:D_optimal_design} 
A Bayesian D-optimal design maximizes the expected KL divergence
between posterior \(\post\) and prior measures \(\pr\). For arbitrary
posterior and prior measures, the KL divergence is defined,
analogously to eq.~\eqref{eq:basic_KL}, via the Radon-Nikodym
derivative:
\begin{equation*}
  D_{KL}(\post||\pr) = \int \log \frac{\der \post}{\der \pr}(\param) \der \post(\param).
\end{equation*}

The study of D-optimal designs for Bayesian linear inverse problems in
infinite dimensions was pioneered by \cite{AlexanderianGloorGhattas14,
  alexanderian2018efficient}. The main result we will make use of is
summarized (in our notation) below:

\begin{theorem}[Alexanderian, Gloor, Ghattas \cite{AlexanderianGloorGhattas14}]\label{thm:d_optimality}
  Let \(\pr = \normal(\prmean,\prcov)\) be a Gaussian prior on \(\hilp\)
  and let \(\post = \normal(\postmean,\postcov)\) the posterior measure
  on \(\hilp\) for the Bayesian linear inverse problem \(\data = \obs
  \fwd\param + \obs \eps' + \eps\) discussed above. Then
  \begin{align}\label{eq:objective}
    \begin{split}
      \tar( \obs) :&= \mathbb{E}_{\data}\left [ D_{\text{KL}} (\post || \pr ) \right ] \\
      &= \frac12 \log \det 
      ( I + \prcov^{1/2}  \fwd ^* \obs^* \Sigma^{-1} \obs \fwd \prcov^{1/2}).
    \end{split}
  \end{align}
\end{theorem}

In \cite{AlexanderianGloorGhattas14, alexanderian2018efficient},
results are stated for \(\Sigma=I\) (implied by \(\modcov =
0,\sigma^2=1\)), but these results also hold for more general
covariance matrices \cite[p. 681]{AlexanderianGloorGhattas14}.

\begin{definition}\label{def:d_optimality}
  We say \(\opt\) is \emph{D-optimal} if \(\opt =
  \argmax_{\obs} \tar(\obs)\), where entries of \(\obs \in (\hilo^*)^m\)
  are constrained to some allowed set of measurements in \(\hilo^*\).
\end{definition}

Intuition for Theorem \ref{thm:d_optimality} can be gained by
recalling from Section \ref{subsec:D} that for a Bayesian linear model
in finite dimensions, with Gaussian prior and Gaussian noise, a
D-optimal design minimizes the determinant of the posterior covariance
matrix. Theorem
\ref{thm:d_optimality} and Definition \ref{def:d_optimality} carry a
similar intuition:
\begin{align*}
  \begin{split}
    \tar(\obs) &= \frac12 \log \det ( I + \prcov^{1/2}  \fwd ^* \obs^* \Sigma^{-1} \obs \fwd \prcov^{1/2}) \text{ by \eqref{eq:objective}}\\
    &= \frac12 \log \det \Big( \prcov ( \prcov^{-1} + \fwd ^* \obs^* \Sigma^{-1} \obs \fwd) \Big )\\
    &= \frac12 \log \det \prcov \postcov^{-1} \text{ by \eqref{eq:postcov}}.
  \end{split}
\end{align*}
We think of \(\prcov\) as constant, in the sense that $\prcov$ does
not depend on data $\data$. Thus, a D-optimal design minimizes a
quantity analogous to the posterior covariance determinant, similarly
to the finite-dimensional case.

\section{The Constrained Optimization Problem of D-Optimal Design}\label{section:D_and_grad}
We seek a formulation of the D-optimal design problem via Lagrange
multipliers. We first find the gradient of $\tar$, then we suggest
unit-norm constraints on $\obs$ and find their gradients. Results of
this section are summarized in Theorem \ref{thm:constrained}. First,
recall that:
\begin{definition}\label{def:var}
  Let $F$ a real valued function of $\obs$. The first variation of $F$
  at $\obs$ in the direction $V$ is:
  \begin{equation*}
    \delta F(\obs) V := \frac{\der}{\der \tau}\Big |_{\tau=0}  F( \obs + \tau V).
  \end{equation*}

  Moreover, if
  \begin{equation*}
    \delta F(\obs) V = \tr{\nabla F(\obs) V},
  \end{equation*}
  then we call $\nabla F(\obs)$ the gradient of $F$ at $\obs$. 
\end{definition}



\begin{proposition}\label{prop:tar_grad}
  The gradient of the D-optimality objective $\tar$ is
  \begin{equation*}
    \nabla \tar(\obs) = ( I - \modcov \obs^* \Sigma^{-1}\obs ) \fwd
    \postcov \fwd^* \obs^* \Sigma^{-1}
  \end{equation*}
\end{proposition}

The proof amounts to calculating the variational derivative of $\tar$
at $\obs$ for any direction $V$ (by Definition \ref{def:var}) and is
delegated to the Supplementary.

\subsection{Unit norm constraints and their gradient}
In a real-life optimal design problem we cannot choose any measurement
operator $\obs \in (\hilo^*)^m$. In order to facilitate analysis, we
seek reasonable constraints on $\obs$ for which finding a D-optimal
design is analytically tractable. The following proposition will guide
us in finding such constraints.

\begin{proposition}\label{prop:bigger_better}
  Let $\obs = (\meas_1,\dots,\meas_m)^t$, $j \in \{1,\dots,m\}$,
  $\sigma^2 > 0$ and $|\zeta| > 1$. Then $\tar(\obs)$ increases if we
  use $\zeta \meas_j$ in $\obs$ instead of $\meas_j$.
\end{proposition}

\begin{proof} 
  Fix $j=1,\dots,m$ and take $V:= e_j e_j^t \obs$. For $u
  \in \hilo$:
  \begin{equation*}
    Vu = e_je_j^t (\meas_1(u),\dots,\meas_m(u) )^t = e_j \meas_j(u)
    = (0,\dots,0,\meas_j(u),0,\dots,0)^t.
  \end{equation*}
  We now calculate the variation of $\tar$ at $\obs$ in the direction
  of $V$. Denote $\tmp: = \fwd \postcov \fwd^*$. From Proposition
  \ref{prop:tar_grad}:
  \begin{align*}
     \delta \tar(\obs) V 
    &= \tr{V ( I - \modcov \obs^*\Sigma^{-1}\obs) \tmp \obs^* \Sigma^{-1}} \\
    &= \tr{e_je_j^t \obs ( I - \modcov \obs^*\Sigma^{-1}\obs) \tmp \obs^* \Sigma^{-1}} \\
    &= e_j^t \obs ( I - \modcov \obs^*\Sigma^{-1}\obs) \tmp \obs^* \Sigma^{-1}e_j \\
    &= e_j^t ( I - \obs \modcov \obs^*\Sigma^{-1})\obs \tmp \obs^* \Sigma^{-1}e_j \\  
    &=  e_j^t(\Sigma-\obs \modcov \obs^*) \Sigma^{-1}\obs \tmp \obs^* \Sigma^{-1}e_j \\
    &=\sigma^2 e_j^t \Sigma^{-1}\obs \tmp \obs^* \Sigma^{-1}e_j
    \text{ by \eqref{eq:Sigma} }\\
    &=\sigma^2 e_j^t \Sigma^{-1}\obs \fwd \postcov \fwd^* \obs^* \Sigma^{-1}e_j.
  \end{align*} 
  Since $\postcov$ is positive definite, we conclude that $\delta
  \tar(\obs) V > 0$. This means that increasing the magnitude of the
  $j^{\text{th}}$ measurement functional increases $\tar(\obs)$.
\end{proof}

Proposition \ref{prop:bigger_better} implies that it is a good idea to
bound the norm of measurements. If, for example, we can take
measurements in $\textup{span}\{\meas\}$ for some $\meas \neq 0$, then
the D-optimality criterion is unbounded, so a D-optimal design does
not exist. In contrast, in any real-life problem where sensors are
concerned, the norm of measurements recorded by sensors is always
one\footnote{Again, our analysis does not directly apply to point
evaluations. We just utilize point evaluations for motivation. We can
approximate point evaluations by e.g.~elements in $\hilo^*$ as long as
$\hilo$ is a function space, e.g.~$L^2(\Omega)$. In this case, for a
fixed apporximation of $\delta$ the norm of the corresponding
functional is a constant $\neq 1$.}:

\begin{equation}
  \| \delta_{\x} \| = \sup_{0 \neq u \in C(\Omega)}
  \frac{
    |\int_{\Omega}u(\y) \delta_{\x}(\y) \der \y|
  }{
    \sup|u|
  } = \sup_{0 \neq u \in C(\Omega)} \frac{|u(\x)|}{\sup|u|} = 1,
  \forall \x \in \Omega.
\end{equation}

Thus, it is reasonable to consider measurements with unit $\hilo^*$
norm. We can write the unit norm constraints as a series of $m$
equality constraints (one for each measurement) on $\obs$. We define
them and find their gradients in Proposition
\ref{prop:constraints_grad} below, whose proof is straightforward and
delegated to the Supplementary:

\begin{proposition}\label{prop:constraints_grad}
  Let
  \begin{align*}
    \phi_j(\obs) :=\frac12 \| \obs^* e_j\|_{\hilp}^2 - \frac12 = 0,\ j=1,\dots,m.
  \end{align*}
  Then
  \begin{equation*}
    \nabla \phi_j(\obs) = \obs^* e_je_j^t.
  \end{equation*}
\end{proposition}




\subsection{Necessary conditions for D-optimality}
We find necessary first-order conditions for D-optimality via Lagrange
multipliers:

\begin{align}
  &\nabla \tar(\obs) = \sum_{j=1}^m \xi_j \nabla \phi_j (\obs)
  \label{eq:Lagrange_mult1} \\
    &\phi_j(\obs) = 0, j = 1,\dots,m. \label{eq:Lagrange_mult2}
\end{align}

We now substitute the gradients calculated in Propositions
\ref{prop:tar_grad} and \ref{prop:constraints_grad} into
eq.~\eqref{eq:Lagrange_mult1}:
\begin{equation}\label{eq:constrained}
  (I - \modcov \obs^* \Sigma^{-1} \obs) \fwd \postcov \fwd^* \obs^*\Sigma^{-1}
  = \sum_{j=1}^m \xi_j \obs^* e_je_j^t = (\xi_1 \meas_1,\dots,\xi_m \meas_m).
\end{equation} 
Letting $\Xi := \diag(\xi_j)$, we can write \eqref{eq:constrained} and
\eqref{eq:Lagrange_mult2} more compactly as:

\begin{theorem}[Necessary conditions for D-Optimality]\label{thm:constrained}
  Let:
  \begin{equation*}
    \opt = \argmax_{\|\meas_j\| = 1, j=1,\dots,m}\tar(\obs).
  \end{equation*}
  
  Then:
  \begin{equation*}
    ( I - \modcov \opt^* \Sigma^{-1} \opt) \fwd \postcov \fwd^* \opt^*  \Sigma^{-1}
    = \opt^* \Xi, 
  \end{equation*}
  where $\Xi \in \mathbb{R}^{m \times m}$ is diagonal.
\end{theorem}

\section{Answer to Question \ref{q:mitigate}: Model error mitigates clusterization}\label{section:non_vanishing}
We now show that if $\modcov \neq 0$ clusterization will not occur. It
is known that including a model error term mitigates the
clusterization phenomenon \cite{Ucinski05}, and here we prove this
rigorously. Let $\obs = (\meas_1,\dots,\meas_m)^t$ and $\obsm :=
(\meas_1,\dots,\meas_{m-1})^t$. Denote $\Sigmam := \Sigma (\obsm)$ and
$\postcovm$ the posterior covariance that arises when $\obsm$ is
utilized as a measurement operator.

\begin{proposition}[Increase due to a measurement]\label{prop:design_increase}
  Let $\obs = (\meas_1,\dots,\meas_m)^t$ and $\obsm :=
  (\meas_1,\dots,\meas_{m-1})^t$. Then
  \begin{equation}\label{eq:conclusion}
    \tar( \obs ) - \tar (\obsm ) =
    \frac12 \log \left ( 1 + \frac{
      \langle \fwd \postcovm \fwd^* (\obsm^* \Sigmam^{-1} \modcov - I ) \meas_m,
      (\obsm^* \Sigmam^{-1} \modcov - I ) \meas_m \rangle
    }{
      \sigma^2 + \meas_m \modcov \meas_m - \meas_m \modcov \obsm^* \Sigmam^{-1} \obsm \modcov \meas_m 
    }       
    \right ).
  \end{equation}
\end{proposition}
The proof is long and tedious, and is delegated to the Supplementary.

\begin{corollary}\label{cor:same_meas}
  If $\meas_m = \meas_j$ for some $1 \leq j \leq m-1$, then
  \begin{equation*}
    \tar(\obs) - \tar(\obsm) =
    \log \left ( 1 + \frac{\sigma^2
      \langle \fwd \postcovm \fwd^* \obsm^* \Sigmam^{-1} e_j,
      \obsm^* \Sigmam^{-1}e_j \rangle
    }{
      2 - \sigma^2 e_j^t\Sigmam^{-1}e_j 
    }       
    \right ),
  \end{equation*}
  where $e_j\in \mathbb{R}^{m-1}$ is the $j^{\text{th}}$ standard unit
  vector.
\end{corollary}

\begin{proof} \label{cor:same_meas_proof}
  Denote $A:= \obs \modcov \obs^*$ and $v_j$ the $j^{\text{th}}$
  column of $A$.  Note that $v_j = \obsm \modcov \meas_m$, since
  $(\obsm \modcov \obsm^*)_{ij} = \meas_i(\modcov \meas_j)$, as
  explained in the paragraph preceding
  eq.~\eqref{eq:modcov_explained}. We can now verify that
  \begin{equation}\label{eq:observation}
    \Sigmam^{-1} \obsm \modcov \meas_m = \Sigmam^{-1}v_j = (A +\sigma^2I_{m-1})^{-1} v_j =
    e_j -\sigma^2 \Sigmam^{-1}e_j.
  \end{equation}
  Using \eqref{eq:observation}:
  \begin{align}\label{eq:denominator}
    \begin{split}
      \meas_m \modcov \obsm^* \Sigmam^{-1} \obsm \modcov \meas_m
      &= \meas_m \modcov \obsm^* ( e_j - \sigma^2 \Sigmam^{-1} e_j )\\
      &= \meas_m \modcov \meas_j - \sigma^2 \meas_m \modcov \obsm^* \Sigmam^{-1}e_j \\
      &= \meas_m \modcov \meas_j -\sigma^2 (e_j - \sigma^2 \Sigmam^{-1}e_j)^t e_j \\
      &= \meas_m \modcov \meas_m -\sigma^2 + \sigma^4 e_j^t\Sigmam^{-1}e_j.
    \end{split}
  \end{align}
  We use \eqref{eq:observation} to simplify the enumerator in
  \eqref{eq:conclusion}:
  \begin{align}\label{eq:enumerator}
    \begin{split}
      (\obsm^* \Sigmam^{-1} \obsm \modcov - I ) \meas_m
      &= \obsm^* \Sigmam^{-1} \obsm \modcov \meas_m - \meas_m \\
      &= \obsm^* (e_j - \sigma^2 \Sigmam^{-1} e_j) -\meas_j \\ 
      &= -\sigma^2 \obsm^* \Sigma^{-1}e_j. 
    \end{split}
  \end{align}
  Now, we substitute \eqref{eq:enumerator} and \eqref{eq:denominator}
  to the enumerator and denominator of \eqref{eq:conclusion}:
  \begin{align*}
    \tar( \obs ) - \tar (\obsm ) &=
    \log \left ( 1 + \frac{
      \langle \fwd \postcovm \fwd^* (\obsm^* \Sigmam^{-1} \modcov - I ) \meas_m,
      (\obsm^* \Sigmam^{-1} \modcov - I ) \meas_m \rangle
    }{
      \sigma^2 + \meas_m \modcov \meas_m - \meas_m \modcov \obsm^* \Sigmam^{-1} \obsm \modcov \meas_m 
    }       
    \right ) \\
    &= \log \left ( 1 + \frac{\sigma^4
      \langle \fwd \postcovm \fwd^* \obsm^* \Sigmam^{-1} e_j,
      \obsm^* \Sigmam^{-1}e_j \rangle
    }{
      2\sigma^2 - \sigma^4 e_j^t\Sigmam^{-1}e_j 
    }       
    \right ) \\
    &= \log \left ( 1 + \frac{\sigma^2
      \langle \fwd \postcovm \fwd^* \obsm^* \Sigmam^{-1} e_j,
      \obsm^* \Sigmam^{-1}e_j \rangle
    }{
      2 - \sigma^2 e_j^t\Sigmam^{-1}e_j 
    }       
    \right ).
  \end{align*}
\end{proof}

Recall from \eqref{eq:Sigma} that $\Sigma(\obs) = \obs
\modcov \obs^* + \sigma^2I$ and let $u := \obsm^*
\Sigmam^{-1}e_j$. Then

\begin{align*}
  \begin{split}
    \lim_{\sigma^2 \to 0} u &= \obsm^*(\obsm \modcov \obsm^*)^{-1}e_j\\
    \lim_{\sigma^2 \to 0} \postcovm &= (\prcov^{-1} + \fwd^* \obsm^* (\obsm \modcov \obsm^*)^{-1} \obsm \fwd)^{-1} \text{ (From \eqref{eq:postcov})}.
  \end{split}
\end{align*}

Consequently, 
\begin{equation*}
   \langle \fwd \postcovm \fwd^* \obsm^* \Sigmam^{-1}
    e_j, \obsm^* \Sigmam^{-1}e_j \rangle 
  = \langle \fwd \postcovm \fwd^* u, u \rangle
\end{equation*}

is bounded, and

\begin{equation*}
\lim_{\sigma^2 \to 0} \tar(\obs) -\tar(\obsm) = 0.
\end{equation*}

We have shown that in the limit $\sigma^2 \to 0$, no increase in
$\tar$ is gained by repeating a measurement. Thus, for vanishing noise
levels, clustered designs cannot be D-optimal. For repeated
measurements and $\sigma^2=0$, $\Sigma$ is not invertible and the
posterior covariance is not defined in eq.~\eqref{eq:postcov}. We can
\emph{define} the posterior in this case to equal the posterior when
the repeated measurement is dropped, and under this definition, a
repeated measurement trivially does not increase the design criterion
when $\sigma^2=0$. Our results are stronger, since we show
\emph{continuity} in $\sigma^2$.

It is worth noting that by the nonnegativity of the KL divergence,
$\tar$ cannot decrease upon adding measurements. However, we can
construct examples where the posterior does not change upon taking a
new measurement, e.g.~if the prior variance vanishes on some
eigenvector and a measurement is taken on said eigenvector. We do not
expect a measurement to generate no information gain whatsoever in any
realistic scenario, and ignore such pathologies.

In conclusion, for small observation error $\sigma^2$ levels,
measurement clusterization is mitigated by the presence of a non-zero
model error $\modcov$ --- answering Question \ref{q:mitigate} posed in
the Introduction.

\section{D-Optimal Designs Without Model Error}\label{section:vanishing}
Our goal in this section is to prove Theorem \ref{thm:char} which
characterizes D-optimal designs when $\modcov = 0$. The necessary
first-order condition for D-optimality of Theorem
\ref{thm:constrained} for $\modcov = 0$ become:

\begin{equation}\label{eq:eigenproblem}
  \sigma^{-2}\fwd \postcov \fwd^* \obs^* = \obs^* \Xi,
\end{equation}
with $\Xi$ diagonal. Equation \eqref{eq:eigenproblem} looks like an
eigenvalue problem for the self-adjoint operator $\sigma^{-2}\fwd
\postcov \fwd^*$, where rows of $\obs$, namely $\meas_j,j=1,\dots, m$,
are eigenvectors. However, $\postcov$ depends on $\obs$, so we refer
to \eqref{eq:eigenproblem} as a \emph{nonlinear} eigenvalue problem.

\begin{proposition}\label{prop:twice_woodbury}
  Assume $\fwd \prcov \fwd^*$ is invertible. Then
  \begin{align*}
    \begin{split}
      \fwd( \prcov^{-1} + \sigma^{-2}  \fwd^* \obs^* \obs \fwd )^{-1} \fwd^* 
      = \left ( (\fwd\prcov\fwd^*)^{-1} + \sigma^{-2}  \obs^* \obs \right )^{-1},
    \end{split}
  \end{align*}  
\end{proposition}

As we mentioned in Section \ref{subsec:bayesian_inverse_problems},
$\fwd \prcov \fwd^*$ is the prior covariance in $\hilo$ and could be
safely assumed invertible. The proof of Proposition
\ref{prop:twice_woodbury} is delegated to the Supplementary. It
amounts to using Woodbury's matrix identity twice. The standard proof
for Woodbury's matrix identity works for separable Hilbert spaces, as
long as all terms are well defined. Unfortunately, $\obs^*\obs$ is not
invertible, so we utilize a regularization trick to force it to be.

\begin{lemma}[Simultaneous diagonalizability]\label{lemma:sim_diag}
  Let $\hil$ separable Hilbert space, $C:\hil \to \hil$ self-adjoint
  and $\func_1,\dots,\func_m \in \hil$. Denote $\func^*$ the element
  $\func$ acting as a linear functional. If
  \begin{equation*}
   (C + \sum_{j=1}^m \func_j\func_j^*) \func_l = \xi_l \func_l,\ l = 1,\dots,m
  \end{equation*}
  then $C$ and $\sum_{j=1}^m \func_j \func_j^*$ are simultaneously
  diagonalizable.
\end{lemma}
The proof of Lemma \ref{lemma:sim_diag} is delegated to the
Supplementary.

\begin{proposition}\label{prop:same_ev}
  Let $\obs$ satisfy the nonlinear eigenvalue problem
  \eqref{eq:eigenproblem}. Then $\obs^*\obs$ and $\fwd \prcov \fwd^*$
  are simultaneously diagonalizable.
\end{proposition}
\begin{proof}
  \begin{align}\label{eq:mod_conditions}
    \begin{split}
      \obs^* \Xi &= \sigma^{-2}\fwd \postcov \fwd^* \obs^*  \text{ (by \eqref{eq:eigenproblem})}\\
      &= \sigma^{-2} \fwd( \prcov^{-1} + \sigma^{-2}  \fwd^* \obs^* \obs \fwd )^{-1} \fwd^* \obs^*  \text{ (by \eqref{eq:postcov})} \\
      &= \sigma^{-2} \left ( (\fwd\prcov\fwd^*)^{-1} + \sigma^{-2}  \obs^* \obs \right )^{-1} \obs^* \text{ (by Proposition \ref{prop:twice_woodbury})}.
    \end{split}
  \end{align}

  Now take $\func_j^{*} = \meas_j$ and $C := (\fwd \prcov
  \fwd^*)^{-1}$ and use Lemma \ref{lemma:sim_diag}.
\end{proof}

Since we made no assumption regarding the ordering of $\{\lambda_i\}$,
we can denote the corresponding non-zero eigenvalues of $\obs^*\obs$
by $\{\eta_i\}_{i=1}^{k}$ and let $\eta_i = 0$ for $i \geq k+1$.

\begin{proposition}\label{prop:true_target}
  Let $\obs$ with $m$ measurements satisfy the nonlinear eigenvalue
  problem \eqref{eq:eigenproblem}. Let $\{\eta_i\}_{i=1}^{\infty}$
  eigenvalues of $\obs^*\obs$ and $\{\lambda_i\}_{i=1}^{\infty}$ the
  corresponding eigenvalues of $\fwd \prcov \fwd^*$. Let $k:=\rank
  \obs^*\obs$. Without loss of generality, let $\eta_i > 0$ for $i\leq
  k$ and $\eta_i = 0$ for $i > k$. Then:
  \begin{enumerate}
    \item $k \leq m$ and $\obs^*\obs$ has exactly $k$ positive
      eigenvalues.
    \item
      \begin{equation*}
        \tar(\obs) = \frac12 \sum_{i=1}^{k} \log (1 + \sigma^{-2}\lambda_i\eta_i) = \frac12 \sum_{i=1}^{m} \log (1 + \sigma^{-2}\lambda_i\eta_i).
      \end{equation*}
    \item Furthermore, if $\obs$ is D-optimal, $\eta_i > 0$ for
      eigenvectors corresponding to the $k$ largest $\lambda_i$.
  \end{enumerate}
\end{proposition}
\begin{proof}
  Part (1) is trivial. To see part (2) holds: 
  \begin{align}
    \begin{split}
      \tar(\obs) &= \frac12\log \det \left (I + \sigma^{-2} \prcov^{1/2} \fwd ^* \obs^*
      \obs \fwd \prcov^{1/2}\right )\\
      &= \frac12 \log \det \left (I + \sigma^{-2} \obs^* \obs \fwd
      \prcov\fwd^* \right ) \text{ (Sylvester's Determinant
      Theorem)}\\
      &=\frac12 \log \prod_{i=1}^{\infty} ( 1 + \sigma^{-2} \lambda_i\eta_i ) \text{ (Proposition \ref{prop:same_ev})} \\
      %
      %
      %
      %
      %
      %
      &=\frac12 \sum_{i=1}^{k} \log (1 + \sigma^{-2}\lambda_i\eta_i). 
      %
      %
    \end{split}
  \end{align}
  Part (3) holds since $\log$ is increasing and $\eta_i \geq 0$.
\end{proof}

\begin{proposition}\label{prop:kkt}
  Let $\tar: \mathbb{R}^m \to \mathbb{R}$, $\tar(\eta) =
  \frac{1}{2}\sum_{i=1}^m \log (1+\sigma^{-2}\lambda_i \eta_i)$, with
  $\lambda_i > 0$ and $\sigma^{2} > 0$. Then the maximum of $\tar$
  subject to $\eta_i \geq 0$ and $\sum\eta_i = m$ is obtained at
  \begin{equation}
  \eta_i = \begin{cases}
    \frac{m}{k} - \sigma^2 \lambda_i^{-1} + \sigma^2 \frac{1}{k} \sum_{j\in A} \lambda_j^{-1} & i \in A \\
    0 & i \in A^c
  \end{cases}
  \end{equation}
  where $A:= \{1\leq i \leq m: \eta_i > 0\}$ and $A^c = \{1,\dots, m\}
  \backslash A$, and $k = |A|$, the cardinality of $A$.
\end{proposition}

The proof of Proposition \ref{prop:kkt} amounts to utilizing the
Karush-Kuhn-Tucker conditions and is delegated to the
Supplementary. The final ingredient we require for the proof of
Theorem \ref{thm:char} characterizing D-optimal designs is:

\begin{lemma}[Unit norm decomposition]\label{lemma:free}
  Let $M \in \R^{k \times k}$ symmetric positive definite with $\ttr M
  = m$, $m \geq k$. We can find $\func_j \in \R^k,j=1,\dots,m$
  with $\|\func_j\|=1$ and $A = (\func_1,\dots,\func_m)$ such that
  $AA^t = M$.
\end{lemma}

The proof of Lemma \ref{lemma:free} is also delegated to the
Supplementary. It is however important to note that this proof is
constructive; it will allow us to construct D-optimal designs, once we
fully characterize them in Theorem \ref{thm:char} below.

\begin{theorem}\label{thm:char}
  Let:
  \begin{itemize}
    \item The D-optimality design criterion
    \cite{AlexanderianGloorGhattas14}:
    \begin{align*}
      \begin{split}
        \tar(\obs) 
        &= \frac12 \log \det ( I + \sigma^{-2} \prcov^{1/2} \fwd ^*
        \obs^* \obs \fwd \prcov^{1/2}), 
      \end{split}
    \end{align*}
  \item \(\opt\) a D-optimal design operator
    \begin{equation*}
      \opt = \argmax_{\|\meas_j\| = 1, j=1,\dots,m}\tar(\obs),
    \end{equation*}
  \item \(\{\lambda_i\}_{i=1}^\infty\) eigenvalues of
    \(\fwd\prcov\fwd^*\) in decreasing order of magnitude.
  \item \(\{\eta_i\}_{i=1}^\infty\) eigenvalues of \(\opt^*\opt\).
 
  \end{itemize}

  Then:
  \begin{enumerate}
  \item  \(\tr{\opt^*\opt} = m\).
  \item \(\opt^*\opt\) and \(\fwd\prcov\fwd^*\) are simultaneously
    diagonalizable.
  \item \(k := \rank \opt^*\opt \leq m\) and
    \begin{equation*}
      \tar(\opt) = \frac12 \sum_{i=1}^{k} \log (1 + \sigma^{-2}\lambda_i\eta_i). 
    \end{equation*}
  \item
    \begin{equation*}
        \eta_i = \begin{cases}
          \frac{m}{k} - \sigma^2 \lambda_i^{-1} + \sigma^2 \frac{1}{k} \sum_{j=1}^k \lambda_j^{-1} & 1 \leq i \leq k \\
          0 & i > k 
        \end{cases}.
    \end{equation*}
  \item The covariance of the pushforward \(\fwd_{*} \postopt\) is \(\left
    ( (\fwd \prcov \fwd^*)^{-1} + \sigma^{-2} \opt^*\opt \right
    )^{-1}\) and its eigenvalues are
    \begin{equation}\label{eq:cylinders}
      \theta_i =
      \begin{cases}
        \left(\frac{\sum_{j=1}^k \lambda_j^{-1} + \sigma^{-2}m}{k} \right )^{-1} & i \leq k \\
        \lambda_i &  i > k 
      \end{cases}.
    \end{equation}
  \end{enumerate}

  Then:
  \begin{enumerate}
  \item  \(\tr{\obs^*\obs} = m\).
  \item \(\obs^*\obs\) and \(\fwd\prcov\fwd^*\) are simultaneously
    diagonalizable.
  \item \(k := \rank \obs^*\obs \leq m\) and
    \begin{equation*}
      \tar(\obs) = \frac12 \sum_{i=1}^{k} \log (1 + \sigma^{-2}\lambda_i\eta_i). 
    \end{equation*}
  \item
    \begin{equation*}
        \eta_i = \begin{cases}
          \frac{m}{k} - \sigma^2 \lambda_i^{-1} + \sigma^2 \frac{1}{k} \sum_{j=1}^k \lambda_j^{-1} & 1 \leq i \leq k \\
          0 & i > k 
        \end{cases}.
    \end{equation*}
  \item The covariance of the pushforwad \(\fwd_{*} \post\) is \(\left
    ( (\fwd \prcov \fwd^*)^{-1} + \sigma^{-2} \obs^*\obs \right
    )^{-1}\) and its eigenvalues are
    \begin{equation*}
      \theta_i =
      \begin{cases}
        \left(\frac{\sum_{j=1}^k \lambda_j^{-1} + \sigma^{-2}m}{k} \right )^{-1} & i \leq k \\
        \lambda_i &  i > k 
      \end{cases}
    \end{equation*}
  \end{enumerate}
\end{theorem}
\begin{proof}
  Part (1) is immediate for any measurement operator $\obs$ that
  satisfies the unit norm constraint on measurements. Part (2)
  was proved in Proposition \ref{prop:same_ev}. Part (3) was proved in
  Proposition \ref{prop:true_target}.
  
  Part (4) is a consequence of Propositions \ref{prop:true_target} and
  \ref{prop:kkt}, with the caveat that we did not show that finding
  $\opt$ so that $\opt^*\opt$ has the desired eigenvalues is
  feasible. To this end, we utilize Lemma \ref{lemma:free}: let $M =
  \diag(\eta_1, \dots, \eta_k)$, diagonal with respect to the first
  $k$ eigenvectors of $\fwd \prcov \fwd^*$. We take $\opt := A$ from
  Lemma \ref{lemma:free}.
  
  Recall from \eqref{eq:postcov}, that the posterior precision is
  $\postcov^{-1} = \prcov^{-1} + \sigma^{-2}\fwd^*\opt^*\opt\fwd$. The
  first statement in part (5) now follows from Proposition
  \ref{prop:twice_woodbury}, while the second statement follows from
  parts (1) and (4).
\end{proof}

\pgfplotstableread{
  Label     prior  optimal  sub-optimal 
  1         0.2    1.8           1.7
  2         0.8    1.2           0.8
  3         2.2    0             0.5
  4         3.5    0             0.0
}\optimalvsnot

\begin{figure}\label{fig:tikz_clusterization}
  \centering
  \begin{tikzpicture}[scale=0.75]
    \begin{axis}[
        ybar stacked,
        ymin=0,
        ymax=4,
        xtick=data,
        legend style={cells={anchor=east}, legend pos=north west, legend columns=-1},
        reverse legend=false, 
        xticklabels from table={\optimalvsnot}{Label},
        xticklabel style={text width=2cm,align=center},
        legend plot pos=right,
        ylabel={\LARGE precision --- prior and posterior},
        xlabel={\LARGE eigenvector},
      ]
      \addplot [fill=blue!60]  table [y=prior,   meta=Label, x expr=\coordindex] {\optimalvsnot};
      \addplot [pattern=north east lines, pattern color=green!80]  table [y=optimal, meta=Label, x expr=\coordindex] {\optimalvsnot};     
      \addlegendentry[scale=1.4]{$\sigma^2\lambda_i^{-1}$}
      \addlegendentry[scale=1.4]{optimal $\eta_i$s}
    \end{axis}
  \end{tikzpicture}
  \begin{tikzpicture}[scale=0.75]
    \begin{axis}[
        ybar stacked,
        ymin=0,
        ymax=4,
        xtick=data,
        legend style={cells={anchor=east}, legend pos=north west, legend columns=-1},
        reverse legend=false, 
        xticklabels from table={\optimalvsnot}{Label},
        xticklabel style={text width=2cm,align=center},
        legend plot pos=right,
        ylabel={\LARGE precision --- prior and posterior},
        xlabel={\LARGE eigenvector} ,
      ]   
      \addplot [fill=blue!60]  table [y=prior,       meta=Label, x expr=\coordindex] {\optimalvsnot};
      \addplot [pattern=north east lines, pattern color=green!80]  table [y=sub-optimal, meta=Label, x expr=\coordindex] {\optimalvsnot};
      \addlegendentry[scale=1.4]{$\sigma^2\lambda_i^{-1}$}
      \addlegendentry[scale=1.4]{sub-optimal $\eta_i$s}
    \end{axis}
  \end{tikzpicture}
  \caption{A comparison of the eigenvalues of the pushforward
    posterior precision $(\fwd\prcov\fwd^*)^{-1} +
    \sigma^{-2}\obs^*\obs$ for a D-optimal design (left) and a
    sub-optimal design (right). Both designs are allowed $m=3$
    measurements. We assume $\sigma^2=1$ and thus, the blue area has
    accumulated height of $\sigma^{-2}m = 3$ in both panels. The
    D-optimal design (left) increases precision where it is
    lowest. The sub-optimal design (right) does not.}
  \label{fig:optimal_vs_not}
\end{figure}

Part (5) of Theorem \ref{thm:char} gives us deep understanding of
D-optimal designs when $\modcov = 0$: Imagine each eigenvector of the
\emph{precision} $\left(\fwd \prcov \fwd^*\right )^{-1}$ corresponds
to a graduated lab cylinder. 
cylinder $i$ is filled, a-priori, with green liquid of
$\lambda_i^{-1}$ volume units. We are allowed $m$ measurement, so we
have blue liquid of volume $\sigma^{-2}m$ units at our disposal. When
we seek a D-optimal design, we distribute the blue liquid by
repeatedly adding a drop to whatever cylinder currently has the lowest
level of liquid in it, as long as its index $i \leq m$. The result of
such a procedure is that the precision for the first $k$ eigenvectors
is the average of their total aggregated precision $\sum_{j=1}^k
\lambda_j^{-1} + \sigma^{-2}m$, see eq.~\eqref{eq:cylinders} and
Fig.~\ref{fig:optimal_vs_not} for an illustration.

\subsection{Answer to Question \ref{q:why}}\label{subsec:why}
Building on Theorem \ref{thm:char}, we can now give a compelling
explanation to the measurement clusterization we observed for the
inverse problem of the heat equation, see Section
\ref{subsub:clusterization1} below. We also suggest a generic
explanation for clusterization, see Section \ref{subsub:cara}.

\subsubsection{Clusterization in the heat equation}\label{subsub:clusterization1}
Consider $\fwd$ and $\prcov$ from \emph{the inverse problem of the
heat equation}. As before, we denote the eigenvalues of
$\fwd\prcov\fwd^*$ by $\lambda_j$. We input these eigenvalues into our
\emph{generic} model, and find a D-optimal design $\opt$ for our
generic model using Theorem \ref{thm:char}. In our generic model, the
measurements we take are best utilized in reducing uncertainty for the
first $k$ eigenvectors. So, a D-optimal design arising from our
\emph{generic model} completely avoids measuring eigenvectors $k+1$
and above.

Of course, in a real life problem --- such as the inverse problem of
the 1D heat equation --- it is likely impossible to find measurements
for which all eigenvectors $k+1$ and above are zero. However, if the
eigenvalues of $\fwd\prcov\fwd^*$ decay quickly (recall the
square-exponential decay for eigenvalues of the 1D heat equation in
eq.\eqref{eq:decay}), a D-optimal design will try to balance measuring
a small number (i.e.~$k$) of the leading eigenvectors.

The abovementioned balance is explored in
Fig.~\ref{fig:eigenvectors}. We allow $m=4$ measurements in $\Omega =
[0,1]$ and observe that D-optimal measurement locations are clustered
at $x_1 = 0.31$ and $x_2 = 0.69$. Upon close inspection of the scaled
eigenvectors of $\fwd \prcov \fwd^*$, we first observe that
eigenvectors $3$ and above have negligible prior amplitude. Since we
only have $m=4$ measurements at our disposal, we interpret these
results, following Theorem \ref{thm:char}, as implying we should only
care about measuring the first and second eigenvectors. Then, we note
the D-optimal $x_1,x_2$ present a compromise between the amplitude of
the first and second eigenvectors. For example, a measurement at
$x=0.5$ would have ignored the second eigenvector altogether, since
the second eigenvector is zero at $x=0.5$.

Now we can understand measurement clusterization for the inverse
problem of the 1D heat equation. A D-optimal design attempts to
measure the first $k$ eigenvectors of $\fwd \prcov \fwd^*$. But there
may be (spatial) limitations on where these $k$ eigenvectors have
large amplitude. For the inverse problem of the heat equation there
are two spatial locations that present a good compromise between the
amplitudes of the first and second eigenvectors, namely $x_1$ and
$x_2$ --- see Fig.~\ref{fig:eigenvectors}. We have $m=4$ measurements
at our disposal but only two spatial locations that are a good
compromise between the amplitudes of the first and second scaled
eigenvectors. Thus, clusterization arises as a consequence of the
pigeonhole principle.

\begin{figure}\label{fig:eigenvectors}
    \centering
    \includegraphics[width=\textwidth]{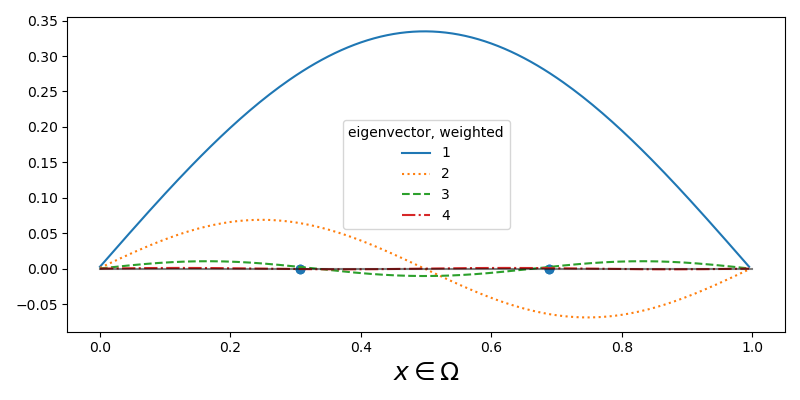}
    \caption{D-optimal measurement locations ($m=4$ measurements) and
      weighted eigenvectors for finding the initial condition of the
      1D heat equation. Measurement locations and weighted
      eigenvectors are plotted over the computational domain $\Omega =
      [0, 1]$ (x-axis). Measurement clusterization occurs
      approximately at $0.31$ and $0.69$. These two locations are a
      compromise between the amplitudes of the first and second
      eigenvectors, which are the eigenvectors that a D-optimal design
      aims to measure. Allocating $m=4$ measurements into two
      locations results in clusterization, according to the pigeonhole
      principle.}
  \label{fig:why}
\end{figure}

\subsubsection{Clusterization in our generic model}\label{subsub:cara}
We can gain further insight to clusterization in our generic model,
from Carath\'eodory's Theorem and the concentration on the first $k$
eigenvectors of $\fwd \prcov \fwd^*$. We present a short discussion
adapting arguments presented in \cite[Chapter 3]{silvey1980} and
\cite[Section 5.2.3]{pronzatoPazman2013}.

Consider $\opt$ a D-optimal design under our generic model.  As
instructed by Theorem \ref{thm:char}, we ignore all but the first $k$
eigenvectors of $\fwd \prcov \fwd$. Thus, we replace $\hilo$ with
$\hilo^{(k)}$ --- the $k$-dimensional subspace spanned by the first
$k$ eigenvectors of $\fwd^*\prcov\fwd$. Let
\begin{equation*}
  \mathcal{M} := \conv \{\meas \meas^* : \meas\in \hilo^{(k)}, \|\meas\|=1\},
\end{equation*}
where $\conv$ denotes the convex hull of a set. The set $\mathcal{M}$
contains only positive-definite operators on a $k$-dimensional vector
space. Hence $\mathcal{M}$ lives in a $k(k+1)/2$-dimensional vector
space. Since $\opt^*\opt = \sum_{i=1}^m \meas_i\meas_i^*$ for $\meas_i
\in \hilo^{(k)}$, it is easy to verify that $\frac1m \opt^*\opt \in
\mathcal{M}$.  Recall Carath\'eodory's Theorem:
\begin{ctheorem}[Carath\'eodory]
  Let $X \subseteq \mathbb{R}^n, X \neq \phi$ and denote $\conv (X)$
  the convex hull of $X$. For every $x \in \conv (X)$, $x$ is a convex
  combination of at most $n+1$ vectors in $X$.
\end{ctheorem}
Carath\'eodory's Theorem implies that there exist $\meas_i$ and
$\alpha_i$ such that
\begin{equation*}
  \opt^*\opt = \sum_{i=1}^I \alpha_i \meas_i\meas_i^*,
\end{equation*}
where $\|\meas_i\|=1, \sum\alpha_i = m, \alpha_i \geq 0$ and $I =
\frac{k(k+1)}{2} + 1$. We can thus write $\opt$ as:

\[
\opt =
\left[
  \begin{array}{ccc}
    \horzbar & \sqrt{\alpha_i} \meas^*_1 & \horzbar \\
    \horzbar & \sqrt{\alpha_2} \meas_2^* & \horzbar \\
             & \vdots    &          \\
    \horzbar & \sqrt{\alpha_I} \meas_I^* & \horzbar \\
  \end{array}
\right].
\]

Unfortunately, $\opt$ is not a valid design, since its rows do not
have unit norm. Still, the above representation of $\opt$ is useful:
If $m > \frac{k(k+1)}{2} + 1$, then $\alpha_i > 1$ for some $1\leq i
\leq I$.  Thus, we can view $\opt$ as a clustered design, since it
places weight $>1$ on a single measurement vector.

\subsubsection{Convergence}\label{subsub:convergence}
In this section we fulfill our promise from Section
\ref{subsub:implications} and prove that the posterior uncertainty
ellipsoid in $\hilo$ will contract to zero along every eigenvector of
$\fwd \prcov \fwd^*$. In our proof we ignore potential problems with
conducting inference on function spaces which are not unique to
D-optimal designs (see e.g.~\cite{owhadi2015} for more details).

First, denote $\opt_m$ a D-optimal design utilizing $m$ measurements
and denote $k_m := \rank\opt_m^*\opt_m$. An immediate consequence of
Theorem \ref{thm:char} is that allowing more measurements will
eventually allow us to measure each eigenvalue, i.e.:
\begin{equation}\label{eq:lim}
  \lim_{m\to\infty} k_m = \infty
\end{equation}

Now, recall that part (5) of Theorem \ref{thm:char} ensures that the
eigenvalues of the posterior pushforward covariance equal
\begin{equation*}
  \theta^{(k)}_i = \left ( \frac{\sum_{j=1}^{k} \lambda_j^{-1} +
    \sigma^{-2}m}{k} \right )^{-1}, \text{ for $i\leq k$}.
\end{equation*}
Using the inequality for the arithmetic and harmonic means, it is easy
to verify that
\begin{equation*}
  \theta^{(k)}_i \leq \lambda_{k}.
\end{equation*}
Since $\lim_{k\to \infty} \lambda_k= 0$, we conclude that for all $i$,
$\lim_{k\to\infty} \theta^{(k)}_i = 0$. Combining the latter
observation with eq.~\eqref{eq:lim}, we conclude that
$\lim_{m\to\infty} \theta^{(k_m)}_i = 0$ for all $i$. Therefore,
posterior uncertainty decays to zero for all eigenvectors.

\section{Numerical Experiments}

\subsection{Simulating Theorem \ref{thm:char}}\label{subsec:lemma_sims}
In the proof of Theorem \ref{thm:char} we utilize Lemma
\ref{lemma:free} to construct D-optimal designs. We implement this
construction with the goal of testing numerically how prevalent are
clustered designs. To this end, we would like to generate random prior
eigenvalues $\lambda_j$, fix $m$ and $k$, find what a D-optimal
$\opt^*\opt$ should be, and then utilize the construction of
Theorem~\ref{thm:char} and Lemma~\ref{lemma:free} to find $\opt$.

To simplify things, we directly generate $\opt^*\opt$. We iterate over
the number of measurements $m \in \{4,\dots, 24\}$, and for every $m$
we then iterate over $k:=\rank \opt^*\opt \in \{2,\dots, m-1\}$. For
each pair $m,k$ we repeat the following steps $N=5000$ times:
\begin{enumerate}
\item Generate random diagonal $D\in \mathbb{R}^{k\times k}$ with
  entries $\log (d_i) \sim \mathcal{N}(50,15)$.
\item Normalize $D$ that $\ttr D = m$.
\item Conjugate $D$ by a random orthogonal matrix to form a positive
  semi-definite $M := UDU^t \in \mathbb{R}^{k\times k}$. This $M$
  represents $\opt^*\opt$.
\item Apply the construction of Lemma \ref{lemma:free} to calculate
  $A$ such that $AA^t = M$, where $A$ has unit norm columns. $A^t$ is
  our optimal design $\opt$.
\item Since $A^t$ corresponds to $\opt$, its columns correspond to
  measurement vectors. We call $A$ "clustered" if $A$ has two or more
  identical columns (up to some numerical precision threshold,
  i.e.~$10^{-5}$).
\end{enumerate}
We then calculate the fraction of clustered designs of the simulations
we ran, for each pair $m,k$. Clusterization occurred at high rates
($>99.9\%$) whenever $m-k > 1$; see Fig.~\ref{fig:sim_AAt}. Hence, in
these simulations, clusterization is a generic property. However, when
$m-k = 1$, clusterization does not occur. We do not why this is so.

\begin{figure}
    \centering
    \includegraphics[height=0.5\textwidth]{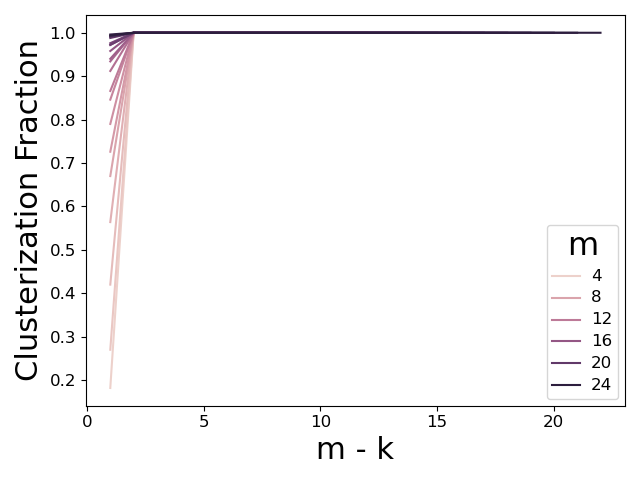}
    \caption{Fraction of clustered $A$ for $AA^t = M$ and $M$
      generated randomly (see text and repository for details on
      generating $M$). It is evident that when $m-k > 1$ clusterization
      is prevalent, whereas for lower $m-k$ clusterization is not.}
  \label{fig:sim_AAt}
\end{figure}

Full results are located in the \texttt{simulations.csv} file within
the accompanying \href{https://github.com/yairdaon/OED}{repository}.
Code implementing the experiments described above is located in module
\texttt{zeros.py} of said repository. Runtime should be less than 30
minutes on any reasonably modern laptop (it took 12 minutes on the
author's laptop).

\subsection{Correlated errors}\label{subsec:corr_errors_sims}
In order to verify the results of Section \ref{section:non_vanishing},
we run simulations of the inverse problem of the 1D heat equation with
nonvanishing model error \(\modcov = \prcov^2 \). Indeed, including
model correlation pushes measurements apart, see
Fig.~\ref{fig:corr_errors}. Code generating Fig.~\ref{fig:corr_errors}
is located in module \texttt{clusterization.py} in the accompanying
\href{https://github.com/yairdaon/OED}{repository}.

\begin{figure}
    \centering
    \includegraphics[height=0.5\textwidth]{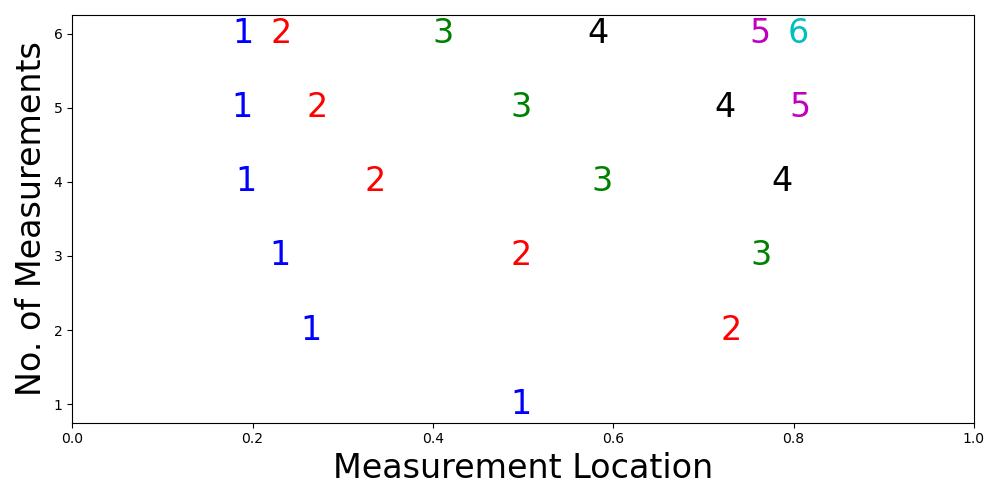}
    \caption{Model correlation mitigates clusterization. We add a
      model correlation term to the error terms in the 1D heat
      equation inverse problem. As expected, measurements are pushed
      away owing to the model error term.}
  \label{fig:corr_errors}
\end{figure}

\begin{acks}[Acknowledgments]
  This study is a result of research I started during my PhD studies
  under the instruction of Prof.~Georg Stadler at the Courant
  Institute of Mathematical Sciences. I would like to thank him for
  his great mentorship, attention to details and kindness. I would
  also like to thank Christian Remling, who helped me find a proof for
  Lemma \ref{lemma:free} in
  \href{https://mathoverflow.net/questions/280168/redistribute-diagonal-entries-of-a-matrix/280203#280203c}{Mathoverflow}.
  Last, but certainly not least, I would like to thank the three
  referees, associate editor and editor in chief for providing
  detailed and insightful reviewes that have made this manuscript a
  whole lot better.
\end{acks}

\begin{funding} 
  This research was supported in part by an appointment with the
  National Science Foundation (NSF) Mathematical Sciences Graduate
  Internship (MSGI) Program sponsored by the NSF Division of
  Mathematical Sciences. This program is administered by the Oak Ridge
  Institute for Science and Education (ORISE) through an interagency
  agreement between the U.S. Department of Energy (DOE) and NSF. ORISE
  is managed for DOE by ORAU. All opinions expressed in this paper are
  the author's and do not necessarily reflect the policies and views
  of NSF, ORAU/ORISE, or DOE.
\end{funding}

\bibliographystyle{plain}


\end{document}